\documentclass[12pt]{amsart}
\usepackage[headings]{fullpage}
\usepackage{amssymb,epic,eepic,epsfig,amsbsy,amsmath,amscd,graphicx,epstopdf}
\usepackage[all]{xy}

\numberwithin{equation}{section}
                        \theoremstyle{plain}
\usepackage{mathrsfs}

\newcommand{\psdraw}[2]
         {\begin{array}{c} \hspace{-1.3mm}
         \raisebox{-4pt}{\psfig{figure=#1.eps,width=#2}}
         \hspace{-1.9mm}\end{array}}

\newtheorem{theorem}{Theorem}[section]
\newtheorem{thm}{Theorem}
\newtheorem{cor}{Corollary}
\newtheorem{lemma}[theorem]{Lemma}

\newtheorem{proposition}[theorem]{Proposition}
\newtheorem{conjecture}{Conjecture}

\newtheorem{definition}{Definition}

\theoremstyle{definition}
\newtheorem{remark}{Remark}

\newcommand\no[1]{}

\newcommand{\lcr}{\raisebox{-5pt}{\mbox{}\hspace{1pt}
                  \epsfig{file=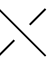}\hspace{1pt}\mbox{}}}
\newcommand{\ift}{\raisebox{-5pt}{\mbox{}\hspace{1pt}
                  \epsfig{file=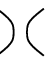}\hspace{1pt}\mbox{}}}
\newcommand{\zer}{\raisebox{-5pt}{\mbox{}\hspace{1pt}
                  \epsfig{file=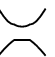}\hspace{1pt}\mbox{}}}

\def\BC{\mathbb C}

\def\BZ{\mathbb Z}

\def\BT{\mathbb T}
\def\BJ{\mathbb J}

\def\CA{\mathcal A}
\def\CK{\mathcal K}

\def\CP{\mathcal P}
\def\CR{\mathcal R}
\def\CS{\mathcal S}
\def\CT{\mathcal T}

\def\fa{\mathfrak a}
\def\fb{\mathfrak b}
\def\fc{\mathfrak c}

\def\fp{\mathfrak p}
\def\ft{\mathfrak t}

\def\fs{\mathfrak s}

\def\bD{{\bar D }}
\def\bTheta{\bar \Theta  }

\def\bbtheta{\bar \theta}

\def\bbp{\bar {\mathfrak p   }}
\def\bbt{\bar {\mathfrak t   }}

\def\bP{\bar \CP}

\def\bT{\bar \CT}
\def\ov{\bar}

\def\la{\langle}
\def\ra{\rangle}

\def\ve{\varepsilon}
\def\be { \begin{equation} }
\def\ee { \end{equation} }

\begin{document}

\title[AJ conjecture]{On the AJ conjecture for cable knots}

\author[Anh T. Tran]{Anh T. Tran}
\address{Department of Mathematical Sciences, University of Texas at Dallas, Richardson, TX 75080, USA}
\email{att140830@utdallas.edu}

\begin{abstract}
We study the AJ conjecture for $(r,2)$-cables of knots, where $r$ is an odd integer. Using skein theory, we show that the AJ conjecture holds true for most $(r,2)$-cables of some classes of two-bridge knots and pretzel knots.
\end{abstract}

\thanks{2010 {\em Mathematics Classification:} Primary 57N10. Secondary 57M25.\\
{\em Key words and phrases: skein module, colored Jones polynomial, character variety, A-polynomial, AJ conjecture, cable knot, two-bridge knot, pretzel knot.}}

\maketitle

\section{Introduction}

\subsection{The colored Jones function} For a knot $K \subset S^3$ and a positive integer $n$, let $J_K(n) \in \BZ[t^{\pm 1}]$ denote the $n$-th colored Jones polynomial of $K$ with framing 0. The polynomial $J_K(n)$ is the quantum link invariant, as defined by Reshetikhin and Turaev \cite{RT}, associated to the Lie algebra $sl_2(\BC)$, with the color $n$ standing for the irreducible $sl_2(\BC)$-representation of dimension $n$. Here we use the functorial normalization, i.e. the one for which $J_{\text{unknot}}(n)= (t^{2n}- t^{-2n})/(t^2 -t^{-2}).$ 

It is known that $J_K(1)=1$ and $J_K(2)$ is the usual Jones polynomial \cite{Jones}. The colored Jones polynomials of higher colors are more or less the Jones polynomials of parallels of the knot. The color $n$ can be assumed to take negative integer values by setting $J_K(-n) = - J_K(n)$ and $J_K(0)=0$. 

We remark that there are different conventions and normalizations of the colored Jones polynomial. In this paper we choose the one which was used in \cite{Le06, LTaj, Tr-twist}.

For a fixed knot $K \subset S^3$, Garoufalidis and Le \cite{GL} proved that the colored Jones function $J_K: \BZ \to \BZ[t^{\pm 1}]$ satisfies a non-trivial linear recurrence relation of the form $\sum_{k=0}^da_k(t,t^{2n})J_K(n+k)=0$, 
where $a_k(u,v) \in \BC[u,v]$ are polynomials with greatest common divisor 1. 

\subsection{Recurrence relations} Let $\CR:=\BC[t^{\pm 1}]$. Consider a function $f: \BZ \to \CR$, and define the linear operators $L$ and $M$ acting on such functions by
$$(Lf)(n) := f(n+1) \quad \text{and} \quad (Mf )(n) := t^{2n} f(n).$$
It is easy to see that $LM = t^2 ML$. The inverse operators $L^{-1}$ and $M^{-1}$ are well-defined. We can consider $L$ and $M$ as elements of the quantum torus
$$ \mathcal T := \mathbb \CR\langle L^{\pm1}, M^{\pm 1} \rangle/ (LM - t^2 ML).$$

The recurrence ideal of $f$ is the left ideal $\CA_f$ in $\CT$ that annihilates $f$:
$$\CA_f:=\{P \in \CT \mid Pf=0\}.$$
We say that $f$ is $q$-holonomic, or $f$ satisfies a non-trivial linear recurrence relation, if $\CA_f \not= 0$. For example, for a fixed knot $K \subset S^3$ the colored Jones function $J_K$ is $q$-holonomic.

\subsection{The recurrence polynomial} Suppose $f: \BZ \to \CR$ is a $q$-holonomic function. Let $\CR(M)$ be the fractional field of the polynomial ring $\CR[M]$. Let $\tilde{\CT}$ be the  set of all Laurent polynomials in the variable $L$ with coefficients in $\CR(M)$:
$$\tilde{\CT} =\left\{\sum_{k \in \BZ} a_k(M)L^k \mid a_k(M) \in \CR(M),~a_k=0\text{~almost always}\right\},$$
and define the product in $\tilde{\CT}$ by $a(M)L^{i} \cdot b(M)L^{k}=a(M)b(t^{2i}M)L^{i+k}$.  

Then it is known that $\tilde{\CT}$ is a principal left ideal domain, and $\CT$ embeds as a subring of $\tilde{\CT}$, c.f. \cite{Ga04}. The ideal extension $\tilde\CA_f:=\tilde\CT\cdot \mathcal A_f$ of $\CA_f$ in $\tilde{\CT}$ is generated by a polynomial 
$$\alpha_f(t,M,L) = \sum_{k=0}^{d} \alpha_{f,k}(t,M) \, L^k,$$
where the degree in $L$ is assumed to be minimal and all the
coefficients $\alpha_{f,k}(t,M)\in \BC[t^{\pm1},M]$ are assumed to
be co-prime. The polynomial $\alpha_f$ is defined up to a polynomial in $\mathbb C[t^{\pm 1},M]$. We call $\alpha_f$ the recurrence polynomial of  $f$.

When $f$ is the colored Jones function $J_K$ of a knot $K$, we let $\CA_K$ and $\alpha_K$ denote the recurrence ideal $\CA_{J_K}$ and the recurrence polynomial $\alpha_{J_K}$ of $J_K$ respectively. We also say that $\CA_K$ and $\alpha_{K}$ are respectively the recurrence ideal and the recurrence polynomial of the knot $K$. Since $J_K(n) \in \BZ[t^{\pm 1}]$, we can assume that $\alpha_K(t,M,L)=\sum_{k=0}^d \alpha_{K,k}(t,M)L^k$
where all the coefficients $\alpha_{K,k} \in \BZ[t^{\pm 1}, M]$ are co-prime.

\subsection{The AJ conjecture}

For a knot $K \subset S^3$, Cooper et al. \cite{CCGLS} introduced the $A$-polynomial $A_K(L,M)$. The $A$-polynomial describes the $SL_2(\BC)$-character variety of the knot complement as viewed from the boundary torus. It carries important information about the topology of the knot. For example, the $A$-polynomial distinguishes the unknot from other knots \cite{BZ, DG}, and the sides of its Newton polygon
give rise to incompressible surfaces in the knot complement [CCGLS].

Motivated by the theory of noncommutative $A$-ideals of Frohman, Gelca and Lofaro \cite{FGL, Ge} and the theory of $q$-holonomicity of quantum invariants of Garoufalidis and Le \cite{GL}, Garoufalidis \cite{Ga04} formulated the following conjecture that relates the $A$-polynomial and the colored Jones polynomial of a knot in $S^3$.

\begin{conjecture}
\label{c1}
{\bf (AJ conjecture)} For every knot $K \subset S^3$, the polynomial $\alpha_K |_{t=-1}$ is equal to $(L-1)A_K(L,M)$, up to a factor depending on $M$ only.
\end{conjecture}

The AJ conjecture has been verified for the trefoil and figure eight knots  \cite{Ga04}, torus knots \cite{Hi, Tr}, some classes of two-bridge knots and pretzel knots \cite{Le06, LTaj}, the knot $7_4$  \cite{GK}, and most cable knots over torus knots \cite{RZ}, over the figure eight knot \cite{Ru, Traj8}, over some two-bridge knots \cite{Dr}, and over twist knots \cite{Tr-twist}. 

\subsection{Main results} Suppose $K$ is a knot with framing 0, and $r,s$ are two integers with $c$ their greatest common divisor. The $(r,s)$-cable $K^{(r,s)}$ of $K$ is the
link consisting of $c$ parallel copies of the $(\frac{r}{c},\frac{s}{c})$-curve on the torus boundary of a tubular neighborhood of $K$. Here an $(\frac{r}{c},\frac{s}{c})$-curve is a curve that is homologically
equal to $\frac{r}{c}$ times the meridian and $\frac{s}{c}$ times the longitude on the torus boundary.
The cable $K^{(r,s)}$ inherits an orientation from $K$, and we assume that each component of $K^{(r,s)}$ has framing 0. Note that if $r$ and $s$ are co-prime, then $K^{(r,s)}$ is again a knot. 

In \cite{Le06, LTaj, Tr-twist}, Thang Le and the author used skein theory to study the AJ conjecture for knots. In this paper, we keep following this geometric approach and study the AJ conjecture for $(r,2)$-cables of knots, where $r$ is an odd integer. Before stating the main results of the paper, we need to introduce the set $\CK$. 

For a Laurent polynomial $f(t) \in \BC[t^{\pm 1}]$, let $d_+[f]$ and $d_{-} [f]$ be respectively the maximal and minimal degree of $f$ in $t^{-4}$. It should be noted that these $d_{\pm}[f]$ are equal to $(-1/4)$ times the corresponding ones in \cite{Le06, LTaj, Tr-twist}.

For a knot $K \subset S^3$, Garoufalidis \cite{Ga-quasi} proved that the degrees of its colored Jones function are quadratic quasi-polynomials. This means that, for each $\ve \in \{\pm\}$ there exist a constant $N_K^{\ve}>0$ and periodic functions $a_{K}^{\ve}(n), b_{K}^{\ve}(n), c_{K}^{\ve}(n)$ such that 
$$d_{\ve}[J_K(n)] = a_{K}^{\ve}(n) \, n^2 + b_{K}^{\ve}(n) n + c_{K}^{\ve}(n)$$
for $n \ge N_K^{\ve}$. 

Following \cite{Ga-slope}, we say that a knot $K$ is \textit{mono-sloped} if $a_{K}^{\ve}(n)$ is a constant function for each $\ve \in \{\pm\}$. In which case we also write $a_{K}^{\ve}$ for $a_{K}^{\ve}(n)$.

\begin{definition}  

With the above notations we define $\CK$ to be the set of all mono-sloped knots $K \subset S^3$ such that $b_{K}^+ (n) \le 0$ and $b_{K}^- (n) \ge 0$. 
\end{definition}

The set $\CK$ contains, for example, adequate knots (by \cite[Lemma 3.2]{Tr-twist}) and $(-2,3, 2m+3)$-pretzel knots (by \cite{KT2}). We remark that the following conjecture is proposed in \cite{KT1}. Note that $b_{U}^+ (n)=1/2$ and $b_{U}^- (n)=-1/2$ for the trivial knot $U$.

\begin{conjecture} \cite{KT1}
\label{KT-conj}
For every non-trivial knot $K \subset S^3$ we have $$b_{K}^+ (n) \le 0 \quad \text{and} \quad b_{K}^- (n) \ge 0.$$
\end{conjecture}

Conjecture \ref{KT-conj} has been verified for several classes of knots, including  adequate knots and their iterated cables, and 2-fusion knots and their cables. We refer the readers to \cite{KT1, KT2} and references therein for more details.

\begin{definition}
For $K \in \CK$ and $\ve \in \{ \pm  \}$, we let 
\begin{eqnarray*}
\delta_{K,1}^{\ve} &=&  \max \big\{|b_{K}^{\ve}(i)-b_{K}^{\ve}(j)| \, : \, i,j \equiv 1 \pmod{2} \big\},\\
\delta_{K,2}^{\ve} &=& \max \big\{\ve (b_{K}^{\ve}(i) + b_{K}^{\ve}(j)) + |c_{K}^{\ve}(i)-c_{K}^{\ve}(j)| \, : \, i,j \equiv 1 \pmod{2} \big\}
\end{eqnarray*}
and $$
r_{K}^{\ve} = 8a_{K}^{\ve} +  \ve \big( 2 \delta_{K,1}^{\ve} + \max \{0, \delta_{K,2}^{\ve} \} \big).$$
\end{definition}

For a knot $K \subset S^3$, we let $X_K$ denote the complement of $K$. That is $X_K = S^3 \setminus K$. Following the definitions and notations in \cite{LTaj}, we state the first result of the paper.

\begin{thm}
\label{main}
Suppose $K$ is a  knot in $\CK$ satisfying all the following conditions:

(i) $K$ is hyperbolic,

(ii) the $SL_2(\BC)$-character variety of $\pi_1(X_K)$ consists of two irreducible components (one abelian and one non-abelian),

(iii) the universal $SL_2(\BC)$-character ring $\fs(X_K)$ is reduced,

(iv) the localized skein module $\overline{\CS(X_K)}$ is finitely generated over $\bar{D}$,

(v) $A_K(L,M) \not= A_K(-L,M)$, and 

(vi) the breadth of the degree of $J_K(n)$ is not a linear function in $n$.

Then the AJ conjecture holds true for the $(r,2)$-cable of $K$ if $r$ is an odd integer satisfying $(r - r_{K}^+)(r - r_{K}^-)>0$. 
\end{thm}

Applying Theorem \ref{main} to two-bridge knots, we obtain the following.

\begin{thm}
\label{main:two-bridge}
Suppose $K$ is a hyperbolic two-bridge knot  such that the non-abelian $SL_2(\BC)$-character variety is irreducible. Then the AJ conjecture holds true for the $(r,2)$-cable of $K$ if $r$ is an odd integer satisfying $(r-4c^+_K)(r+4c^-_K)>0$, where $c^{\pm}_K$ denotes the number of positive/negative crossings of $K$.
\end{thm}

As a corollary of Theorem \ref{main:two-bridge}, we have:

\begin{cor}
The AJ conjecture holds true for the $(r,2)$-cable of the double twist knot $J(k,l)$ (with $k \not= l$) if $r$ is an odd integer satisfying 
$$\begin{cases} r \big( r-4(k+l-1) \big)>0 &\mbox{if } k>0, l>0, \\ 
r \big( r-4(k+l+1) \big)>0 &\mbox{if } k<0, l<0,\\
(r+4k)(r+4l)>0 &\mbox{if } kl < 0. 
\end{cases} $$
\end{cor}

\begin{figure}[htpb]
$$ \psdraw{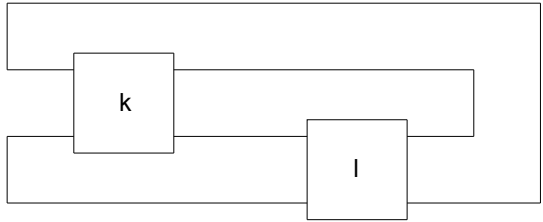}{3.5in} $$
\caption{The double twist knot $J(k, l)$, where $k$ and $l$ denote the numbers of half
twists in the boxes. Positive/negative numbers correspond to right-handed/left-handed twists respectively.}
\end{figure}

Let $K(m)$ denote the $(-2,3,2m+3)$-pretzel knot. It is known that $K(m)$ is a torus knot if $|m| \le 1$, and is a hyperbolic knot if $|m| \ge 2$. By applying Theorem \ref{main} we have:

\begin{thm}
\label{main:pretzel}
The AJ conjecture holds true for the $(r,2)$-cable of the hyperbolic $(-2,3,2m+3)$-pretzel knot (with $|m| \ge 2$ and $\gcd(m,3)=1$) if $r$ is an odd integer satisfying $$\begin{cases} r \big( r - (\frac{33m}{4} + \frac{3}{m} + 19) \big) > 0 &\mbox{if } m \ge 2, \\ 
r(r - 20) >0 &\mbox{if } m = -2,\\
(r  - \big(\frac{33m}{4} + \frac{6}{2m+3} + \frac{171}{8} \big))(r - 20) >0 &\mbox{if } m \le -3.
\end{cases} $$
\end{thm}

\begin{figure}[htpb]
$$ \psdraw{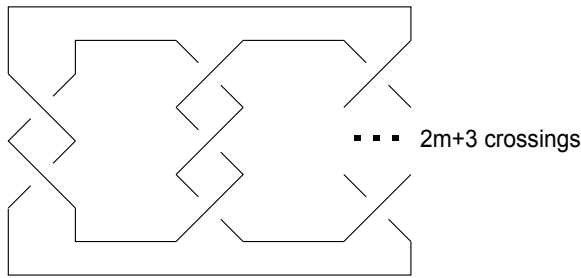}{3in} $$
\caption{The $(-2,3,2m+3)$-pretzel knot.}
\end{figure}

\begin{remark}
$a)$ In \cite{LTaj}, it was proved that if a knot $K \subset S^3$ satisfies conditions $(i)-(iv)$ and condition $(vi)$ of Theorem \ref{main} then $K$ satisfies the AJ conjecture.

$b)$ Regarding Theorem \ref{main}, we \textit{conjecture} that condition $(v)$ follows from conditions $(i)$ and $(ii)$. That is, if $K \subset S^3$ is a hyperbolic knot such that the non-abelian character variety of $\pi_1(X_K)$ is irreducible, then $A_K(L,M) \not= A_K(-L,M)$. This has been shown to hold true for two-bridge knots and $(-2,3,2m+3)$-pretzel knots, see Section 6. 
\end{remark}

\subsection{Plan of the paper} In Section \ref{char_skein} we review character varieties and skein modules. In Section \ref{cable} we study the recurrence polynomial and the $\gamma$-polynomial of a knot. Similarly, in Section \ref{AC} we study the $A$-polynomial and the $C$-polynomial of a knot. In Section \ref{action} we use skein theory to relate the $\gamma$-polynomial and the $C$-polynomial. Finally, we prove the main results, namely Theorems 1--3 and Corollary 1, in Section \ref{twist}.

\section{Character varieties and skein modules} \label{char_skein}

In this section we will review character varieties and skein modules, and their relationship. We will also review the classical and quantum peripheral ideals of a knot. However, we will not review the definitions of the $A$-polynomial and the colored Jones polynomial. We refer the readers to \cite{CCGLS} and \cite{Jones, RT} for those definitions. 

\subsection{Character varieties} 

\label{char}

The set of characters of representations of a finitely generated group $G$ into $SL_2(\BC)$ is known to be a complex algebraic set, called the character variety of $G$ and denoted by $\chi(G)$ (see \cite{CS, LM}). For a manifold $Y$ we also use $\chi(Y)$ to denote $\chi(\pi_1(Y))$. Suppose  $G=\BZ^2$, the free abelian group with 2 generators.
Every pair  of generators $\mu,\lambda$ will define an isomorphism
between $\chi(G)$ and $(\BC^*)^2/\tau$, where $(\BC^*)^2$ is the
set of non-zero complex pairs $(M,L)$ and $\tau: (\BC^*)^2 \to (\BC^*)^2$ is the involution
defined by $\tau(M,L):=(M^{-1},L^{-1})$, as follows. Every representation is
conjugate to an upper diagonal one, with $M$ and $L$ being the
upper left entry of $\mu$ and $\lambda$ respectively. The
isomorphism does not change if we replace $(\mu,\lambda)$ with
$(\mu^{-1},\lambda^{-1})$.

\subsubsection{The classical peripheral ideal} Let $K$ be a knot in $S^3$ and $X_K$ its complement. The boundary of
$X_K$ is a torus whose fundamental group is $\BZ^2$. The inclusion $\partial X_K \hookrightarrow X_K$ induces the restriction
map
$$\rho : \chi(X_K) \longmapsto \chi(\partial X_K)\equiv (\BC^*)^2/\tau.$$

For a complex algebraic set $V$, let $\BC[V]$ denote the ring
of regular functions on $V$.  For example, $\BC[(\BC^*)^2/\tau]=
\ft^\sigma$, the $\sigma$-invariant subspace of  $\ft:=\BC[L^{\pm
1},M^{\pm 1}]$, where $\sigma(M^{k}L^{l})= M^{-k}L^{-l}$. The map $\rho$ above induces an algebra homomorphism
$$\theta: \BC[\chi(\partial X_K)]  \equiv \ft^\sigma \longrightarrow
\BC[\chi(X_K)].$$
We call the kernel $\fp$ of $\theta$ the classical
peripheral ideal; it is an ideal of $\ft^\sigma$. 

\subsubsection{The $B$-polynomial} The ring $\BC[\chi(X_K)]$ has a $\ft^{\sigma}$-module structure. For $f \in \ft^{\sigma}$ and $g \in \BC[\chi(X_K)]$, the action $f \cdot g$ is defined to be $\theta(f)g \in \BC[\chi(X_K)]$. Since  $\BC[M^{\pm 1}]^{\sigma}$ is a subring of $\ft^{\sigma}$, $\BC[\chi(X_K)]$ also has a $\BC[M^{\pm 1}]^{\sigma}$-module structure. Extending the map $\theta: \ft^{\sigma} \to \BC[\chi(X_K)]$ from the ground ring ${\BC[M^{\pm 1}]}^{\sigma}$ to $\BC(M)$, the fractional field of $\BC[M]$, we get
$$
\left( \bar{\ft} \overset{\bar{\theta}}{\longrightarrow} \overline{\BC[\chi(X_K)]} \right) :=\left( \ft^\sigma \overset{\theta}{\longrightarrow} \BC[\chi(X_K)] \right) \otimes _{\BC[M^{\pm 1}]^{\sigma}} \BC(M).
$$

The ring $\bar \ft=\BC(M)[L^{\pm1}]$ is a PID. The ideal $\bar{\fp}:= \ker \bar{\theta} \subset \bar{\ft}$ is thus generated by a single polynomial
$B_K\in \BZ[L,M]$ which has co-prime coefficients and is defined
up to a factor $\pm M^k$ with $k\in \BZ$. We call $B_K$ the $B$-polynomial of $K$. 

From definition, we see that $B_K(L,M)$ is the polynomial in $\bbt$ of minimal $L$-degree such that $\bbtheta(B_K) = 0$. In \cite{LTaj} it was proved that $B_K(L,M) = (L-1)A_K(L,M)$, up to a factor depending on $M$ only, where $A_K(L,M)$ is the $A$-polynomial of $K$ defined in \cite{CCGLS}.

\subsection{Skein modules}  Recall that $\CR=\BC[t^{\pm1}]$. Let $\mathcal{L}$ be the set of isotopy classes of framed links in the oriented manifold $Y$, including the empty
link. Consider the free $\CR$-module with basis $\mathcal{L}$, and
factor it by the smallest submodule containing all Kauffman expressions \cite{Ka} of
the form $\lcr-t\zer-t^{-1}\ift$ and
$\bigcirc+(t^2+t^{-2}) \emptyset$, where the links in each
expression are identical except in a ball in which they look like
depicted. This quotient is denoted by $\CS(Y)$ and is called the
Kauffman bracket skein module, or just skein module, of $Y$ (see \cite{Pr, Tu}). 

For an oriented surface $\Sigma$ we define $\CS(\Sigma):= \CS(Y)$, where $ Y= \Sigma \times [0,1]$ is the cylinder over $\Sigma$. The skein module
$\CS(\Sigma)$ has an algebra structure induced by the operation
of gluing one cylinder on top of the other. The operation of
gluing the cylinder over $\partial Y$ to $Y$ induces a
$\CS(\partial Y)$-left module structure on $\CS(Y)$. 

Let $\BT^2$ be the torus with a fixed pair $(\mu, \lambda)$ of simple closed curves intersecting at exactly one point. For co-prime integers $k$ and $l$, let $\lambda_{k,l}$ be a simple closed curve  on the torus homologically equal to $k\mu+ l\lambda$. Recall that $ \CT= \CR\la M^{\pm 1}, L^{\pm1} \ra/(LM - t^2 ML)$ is the quantum torus. Let $\sigma: \CT \to \CT$ be the involution defined by $\sigma(M^{k} L^{l}) := M^{-k} L^{-l}$.
Frohman and Gelca \cite{FG} (see also \cite{Sa}) showed that there is an algebra isomorphism $ \CS(\BT^2)\to \CT^{\sigma}$ sending $\lambda_{k,l}$ to $(-1)^{k+l} t^{kl} (M^{k}L^{l} +M^{-k}L^{-l}).$

\subsubsection{The quantum peripheral ideal} Recall that $X_K$ is the complement of $K \subset S^3$. The boundary of
$X_K$ is a torus whose skein module is $\CT^{\sigma}$. The operation of
gluing the cylinder over $\partial X_K$ to $X_K$ induces a
$\CT^\sigma$-left module structure on $\CS(X_K)$. For $\ell \in \CT^\sigma=\CS(\partial X_K)$ and $\ell' \in \CS(X_K)$, the action $\ell \cdot \ell' \in \CS(X_K)$ is the disjoint union of $\ell$ and $\ell'$. The map
$$ \Theta: \CS(\partial X_K) \equiv \CT^\sigma \to \CS(X_K), \quad \Theta(\ell) := \ell \cdot \emptyset$$
can be considered as a quantum analog of the map $\theta: \ft^{\sigma} \to \BC[\chi(X_K)]$ defined in Subsection \ref{char}. Its kernel $\CP:=\ker \Theta$ is called the quantum peripheral ideal, first introduced in \cite{FGL}. In \cite{FGL, Ge} (see also \cite{ Ga08}), it was proved that every element in $\CP$ gives rise to a recurrence relation for the colored Jones polynomial. That is $\CP \subset \CA_K$.

\subsection{Skein modules as deformations of universal character rings}

Let $\ve$ be the map reducing $t=-1$. An important result \cite{Bu,PS} in the theory of skein modules is that
$\fs(Y):=\varepsilon (\CS(Y))$ has a natural $\BC$-algebra structure and is isomorphic to the universal $SL_2(\BC)$-character algebra of $\pi_1(Y)$. See \cite{BH, PS} for a definition of the universal character algebra. The product of two links in $\fs(Y)$ is their disjoint union, which  is well-defined when $t=-1$. The isomorphism between
$\fs(Y)$ and the universal $SL_2(\BC)$-character algebra of $\pi_1(Y)$ is given by $K(r)= -\text{tr}\, r(K)$,
where $K$ is a knot in $Y$ representing an
element of $\pi_1(Y)$ and $r:\pi_1(Y) \to
SL_2(\BC)$ is a representation of $\pi_1(Y)$.
The quotient of $\fs(Y)$
 by its nilradical is canonically isomorphic to
$\BC[\chi(Y)]$, the ring of regular functions on the $SL_2(\BC)$-character
variety of $\pi_1(Y)$.

In many cases the nilradical of $\fs(Y)$ is trivial, and hence $\fs(Y)$ is exactly equal to the character ring $\BC[\chi(Y)]$. For example, this is the case when $Y$ is a torus, or when $Y$ is the
complement of a two-bridge knot/link \cite{Le06, PS, LTskein}, or when $Y$ is the
complement of the $(-2,3,2n + 1)$-pretzel knot/the $(-2,2m,2n+1)$-pretzel link \cite{LTaj, Tr-pretzel}.

\section{The recurrence polynomial and $\gamma$-polynomial} \label{cable}

In this section, we will prove some properties of the recurrence polynomial of cables of  knots in the set $\CK$. We will also define and study the $\gamma$-polynomial of a knot, a relative of the recurrence polynomial. 

\subsection{The recurrence polynomial of cable knots} 

Suppose $r$ is an odd integer. For $n>0$, by \cite[Section 2.1]{LTvol} we have
\begin{equation}
\label{cables}
J_{K^{(r,2)}}(n) = t^{-2r(n^2-1)}\sum_{k=1}^n (-1)^{r(n-k)} t^{2rk(k-1)} J_K(2k-1).
\end{equation} 

For a Laurent polynomial $f(t) \in \BC[t^{\pm 1}]$, recall that $d_+[f]$ and $d_{-} [f]$ be respectively the maximal and minimal degree of $f$ in $t^{-4}$.

\begin{lemma} 

\label{deg}

Suppose $K$ is a knot in $\CK$.

(a) If $r > r_K^+$, then for $n \gg 0$ we have
$$d_+[J_{K^{(r,2)}}(n)] = r(n^2-1)/2.$$

(b) If $r < r_K^-$, then for $n \gg 0$ we have
$$d_-[J_{K^{(r,2)}}(n)] = r(n^2-1)/2.$$
\end{lemma}

\begin{proof}
Fix $n \gg 0$. We prove the formula for $d_+[J_{K^{(r,2)}}(n)]$. The one for $d_-[J_{K^{(r,2)}}(n)]$ is proved similarly. By equation \eqref{cables} we have
\begin{equation}
\label{d+}
d_+[J_{K^{(r,2)}}(n)] = r(n^2-1)/2 +  d_+ \big[ \sum_{k=1}^n (-1)^{r(n-k)} t^{2rk(k-1)} J_K(2k-1) \big].
\end{equation}
In the above formula, there is a sum. Under the assumption $r > r_K^+$, we will show that there is a unique term of the sum whose highest degree is strictly greater than those of the other terms. This implies that the highest degree of the sum is exactly equal to the highest degree of that unique term. 

For $k \in \{1, 2, \dots, n\}$ let 
\begin{eqnarray*}
f(k) &:=& d_+ [(-1)^{r(n-k)} t^{2rk(k-1)} J_K(2k-1)] \\
     &=& -rk(k-1)/2 + d_+[J_K(2k-1)].
\end{eqnarray*}
Denote $a_K^+$, $b^+_K(n)$, $c^+_K(n)$, $\delta^+_{K,1}$ and $\delta^+_{K,2}$ by $\fa$, $\fb_n$, $\fc_n$, $\delta_1$ and $\delta_2$ respectively. Then 
\begin{eqnarray*}
\delta_1 &=&  \max \big\{|\fb_i - \fb_j| \, : \, i,j  \equiv 1 \pmod{2} \big\},\\
\delta_2 &=& \max \big\{\fb_i +  \fb_j + |\fc_i - \fc_j| \, : \, i,j \equiv 1 \pmod{2} \big\}
\end{eqnarray*}
and $$r^+_{K} = 8 \fa +  2 \delta_1 + \max \{0, \delta_2 \}.$$

We have
\begin{eqnarray*}
f(k)&=& -rk(k-1)/2 + \fa (2k-1)^2 + \fb_{2k-1} (2k-1) + \fc_{2k-1}\\
    &=& - (r/2 - 4\fa) k^2 + (r/2 - 4 \fa + 2 \fb_{2k-1}) k + \fa - \fb_{2k-1} + \fc_{2k-1}.
\end{eqnarray*}
We claim that if $k_1 > k_2$ then $f(k_1) < f(k_2)$. Indeed, let $\omega = f(k_1) - f(k_2)$. Then $\omega = \omega' + ( \fc_{2k_1 -1} - \fc_{2k_2 -1} )$ where
$$
\omega' := - (r/2 - 4\fa) (k_1 - k_2) (k_1 + k_2 -1)  + 2 \fb_{2k_1-1} (k_1 - k_2) + (\fb_{2k_1-1} - \fb_{{2k_2-1}}) (2k_2 -1).
$$

Since $r/2 - 4\fa > r^+_K/2 - 4\fa \ge 0$, $k_1 > k_2 \ge 1$ and $\fb_{2k_1-1} \le 0$ we have
\begin{eqnarray*}
- (r/2 - 4\fa) (k_1 - k_2) (k_1 + k_2 -1)  &\le& - (r/2 - 4\fa) (2k_2),\\
2 \fb_{2k_1-1} (k_1 - k_2) &\le& 2 \fb_{2k_1-1}.
\end{eqnarray*}
Hence 
\begin{eqnarray*}
\omega' &\le& - (r/2 - 4\fa) (2 k_2)  + 2 \fb_{2k_1-1}  + (\fb_{2k_1-1} - \fb_{{2k_2-1}}) (2k_2 -1) \\
 &=&  - \big( r/2 - 4 \fa - (\fb_{2k_1-1} - \fb_{{2k_2-1}}) \big) (2k_2)  +  (\fb_{2k_1-1} + \fb_{{2k_2-1}}).
\end{eqnarray*}

Since $r/2 - 4 \fa - (\fb_{2k_1-1} - \fb_{{2k_2-1}}) >  r^+_K/2 - 4\fa - \delta_1 = \max\{0, \delta_2/2\}$ we have
\begin{eqnarray*}
\omega' &<&  - \big( r^+_K - 8\fa - 2\delta_1 \big) +  (\fb_{2k_1-1} + \fb_{{2k_2-1}}) \\
               &\le& - \delta_2 + (\fb_{2k_1-1} + \fb_{{2k_2-1}})\\
        &\le& - |c_{2k_1-1} - c_{2k_2-1}|.  
\end{eqnarray*}
Hence $\omega = \omega' + ( \fc_{2k_1 -1} - \fc_{2k_2 -1} ) < 0$. This means that $f(k_1) < f(k_2)$ for $k_1 > k_2$. In particular, $f(k)$ attains its maximum on the set $\{1, 2, \dots, n\}$ at a unique $k=1$. Hence, equation \eqref{d+} implies that $$d_+[J_{K^{(r,2)}}(n)] = r(n^2-1)/2 + f(1).$$
Since $f(1)= d_+[J_K(1)] = 0$, the lemma follows.
\end{proof}

Let $\BJ_K(n):=J_K(2n+1)$. Then
\begin{equation}
\label{f}
M^r(L+t^{-2r}M^{-2r})J_{K^{(r,2)}} = \BJ_K,
\end{equation} 
see \cite[Section 6.1]{RZ} or \cite[Lemma 2.1]{Traj8}.

\begin{proposition}
\label{divisible}
Suppose $K$ is a knot in $\CK$. If $r$ is an odd integer with $(r-r_{K}^+)(r-r_{K}^-) > 0$, then $$\alpha_{K^{(r,2)}}=\alpha_{\BJ_K}M^r(L+t^{-2r}M^{-2r}).$$ 
\end{proposition}

\begin{proof}
The proof is similar to that of \cite[Proposition 3.3]{Tr-twist}. We first claim that $\alpha_{K^{(r,2)}}$ is left divisible by $M^r(L+t^{-2r}M^{-2r})$. Indeed, write $$\alpha_{K^{(r,2)}}=QM^r(L+t^{-2r}M^{-2r})+R$$ where $Q \in \CR[M^{\pm 1}][L]$ and $R \in \CR[M^{\pm 1}]$. 

From equation \eqref{f} we have
\begin{eqnarray}
\label{div}
0 &=& \alpha_{K^{(r,2)}} J_{K^{(r,2)}} \nonumber\\
  &=& QM^r(L+t^{-2r}M^{-2r})J_{K^{(r,2)}}+RJ_{K^{(r,2)}}\nonumber\\
  &=& Q \BJ_K + RJ_{K^{(r,2)}}.
\end{eqnarray}

Assume that $R \not= 0$. Write $Q=\sum_{k=0}^d e_k(M) L^k$ where $e_k(M) \in \CR[M^{\pm 1}]$. Suppose $r > r_{K}^+ = 8a_{K}^+ + 2\delta_{K,1}^+ + \max\{0,  \delta_{K,2}^+ \}$. For $n \gg 0$, by Lemma \ref{deg} we have
\begin{equation}
\label{eq11}
d_+[R J_{K^{(r,2)}}(n)] = d_+[J_{K^{(r,2)}}(n)] + O(n) =rn^2/2+O(n).
\end{equation}
This implies that $R J_{K^{(r,2)}} \not=0$. Hence, by equation \eqref{div}, we have $Q \not=0$. 

Let $I = \{0 \le k \le d \mid e_k \not= 0\}$. If $k \in I$ then for $n \gg 0$ we have
\begin{equation}
\label{eq12}
d_+ [(e_k L^k \BJ_K)(n)]= d_+[J_K(2n+2k+1)]+O(n) = 4a_{K}^+ \, n^2 + O(n).
\end{equation}
Since $r/2 > r_{K}^+ /2 \ge 4a_{K}^+$, equations \eqref{eq11} and \eqref{eq12} imply that for $n \gg 0$ we have
$$d_+[RJ_{K^{(r,2)}}(n)] > \max_{k \in I} \{ d_+[(e_k L^k \BJ_K)(n)] \} \ge d_+[Q \BJ_K(n)].$$
This contradicts equation \eqref{div}. 

Similarly, if $r < r_{K}^- $ then for $n \gg 0$ we have
$$d_{-}[RJ_{K^{(r,2)}}(n)] < \min_{k \in I}  \{ d_{-} [(e_k L^k \BJ_K)(n)] \} \le d_{-} [Q \BJ_K(n)].$$ This also contradicts equation \eqref{div}. 

Hence $R=0$, which means $\alpha_{K^{(r,2)}}$ is left divisible by $M^r(L+t^{-2r}M^{-2r})$. Then, since $M^r(L+t^{-2r}M^{-2r})J_{K^{(r,2)}} = \BJ_K$,
it is easy to see that $\alpha_{K^{(r,2)}}=\alpha_{\BJ_K}M^r(L+t^{-2r}M^{-2r})$.
\end{proof}

To determine $\alpha_{K^{(r,2)}}$, we will need the following lemma.

\begin{lemma}
\label{L2}
\cite[Lemma 4.4]{Tr-twist} For $P(L,M) \in \CT$ we have
$$\left( P(L^2,M)J_K \right)(2n+1) = \left( P(L,t^2M^2)\BJ_K \right)(n).$$
\end{lemma}

\subsection{The $\gamma$-polynomial} 

\label{gamma}

To determine $\alpha_{\BJ_\CK}$ (and hence $\alpha_{K^{(r,2)}}$), by Lemma \ref{L2} we need to find an element $P(L,M) \in \CT$ of minimial $L$-degree such that $P(L^2,M)$ annihilates the colored Jones function $J_{K}(n)$. For this purpose we now define and study a relative of the recurrence polynomial, namely the $\gamma$-polynomial.

Let $\ell := L^2$. Let
$$ \mathcal T_2 := \mathbb \CR\langle \ell^{\pm1}, M^{\pm 1} \rangle/ (\ell M - t^4 M \ell)$$
be a subring of the quantum torus $\CT$, and $\Theta_2$ the restriction of $\Theta: \CT^{\sigma} \to \CS(X_K)$ on $\CT_2^{\sigma} \subset \CT^{\sigma}$. That is
$$ \Theta_2 = \big( \Theta \mid_{\CT_2^{\sigma}} : \CT_2^{\sigma} \to \CS(X_K) \big).$$
Let $\CP_2 : = \ker \Theta_2 \subset \CT_2$.

Recall that $\CR(M)$ is the fractional field of the polynomial ring $\CR[M]$. The ring $\CT_2^{\sigma} \subset \CT^{\sigma}$ is an $\CR[M^{\pm 1}]^{\sigma}$-algebra. Let $\tilde \CT_2 :=\CT_2^{\sigma} \otimes_{\CR[M^{\pm 1}]^{\sigma}} \CR(M)$. We have
$$\tilde{\CT}_2 =\left\{\sum_{k \in \BZ} a_k(M) \ell^k \mid a_k(M) \in \CR(M),~a_k=0\text{~almost always}\right\}$$
with commutation rule $a(M) \ell^{i} \cdot b(M) \ell^{k}=a(M)b(t^{4i}M) \ell^{i+k}$. 

It is known that $\tilde{\CT}_2$ is a left PID. The ideal extension $\tilde{\CP}_2 := \CP_2 \cdot \tilde{\CT}_2 \subset \tilde{\CT}_2$ is a principal left ideal, and thus can be generated by a polynomial $\gamma_K(\ell,M) \in \CR[\ell, M]$. We call $\gamma_K(\ell,M)$ the $\gamma$-polynomial of $K$.

From definition, we see that $\gamma_K(L^2,M)$ is an element in the quantum peripheral ideal $\CP$. Since $\CP \subset \CA_K$, we have $\gamma_K(L^2,M) \in \CA_K$. This means that $\gamma_K(L^2,M)$ annihilates the colored Jones function $J_K(n)$.

\section{The $A$-polynomial and $C$-polynomial} \label{AC} 

In this section, we will prove some properties of the $A$-polynomial of cable knots. We will also define and study the $C$-polynomial of a knot, a relative of the $A$-polynomial.

\subsection{The $A$-polynomial} A formula for the $A$-polynomial of cable knots has recently been given by Ni and Zhang, c.f. \cite{Ru}. In particular, for an odd integer $r$ we have 
\begin{equation}\label{A_cable}A_{K^{(r,2)}}(L,M)=
(L-1) \text{Res}_{\lambda} \left( A_{K}(\lambda,M^2),\lambda^2-L \right)(L+M^{-2r})
\end{equation}
where $\text{Res}_{\lambda}$ denotes the polynomial resultant eliminating the variable $\lambda$.

%We now prove some properties of the $A$-polynomial. Note that $A_K(L,-M)=A_K(L,M)$.

\begin{lemma}
\label{even_M}
Suppose $P(L,M) \in \BC[L,M]$ is irreducible and $P(L,-M)=P(L,M)$. Then $P(L,M^2) \in \BC[L,M]$ is irreducible.
\end{lemma}

\begin{proof}
Let $Q(L,M) := P(L,M^2).$ Since $P(L,-M)=P(L,M)$, we have
\begin{equation}
\label{eqQ}
Q(L,M\sqrt{-1})=P(L,-M^2)=P(L,M^2)=Q(L,M).
\end{equation}
%where $i = \sqrt{-1}$ denotes the imaginary unit.

Assume that $Q(L,M)$ is reducible in $\BC[L,M]$. Since $P(L,M)$ is irreducible in $\BC[L,M]$, we have $Q(L,M)=R(L,M)R(L,-M)$, for some irreducible $R(L,M) \in \BC[L,M]$. Equation \eqref{eqQ} is then equivalent to
$$R(L,M\sqrt{-1})R(L,-M\sqrt{-1}) = R(L,M)R(L,-M).$$
This implies that $R(L,M)=\pm R(L,M\sqrt{-1})$ or $R(L,M)=\pm R(L,-M\sqrt{-1})$. 

Without loss of generality, we can assume that $R(L,M)=\pm R(L,M\sqrt{-1})$. Then
$$R(L,M)=\pm R(L,M\sqrt{-1}) = R(L,(M\sqrt{-1})\sqrt{-1}) = R(L,-M).$$
This implies that $R(L,M) = S(L,M^2)$ for some $S(L,M) \in \BC[L,M]$. Hence
$$P(L,M^2) = R(L,M)R(L,-M) = (S(L,M^2))^2$$
which means $P(L,M) = (S(L,M))^2$. This contracts the assumption that $P(L,M) \in \BC[L,M]$ is irreducible. Hence $Q(L,M) \in \BC[L,M]$ is irreducible.
\end{proof}

 Let 
$$R_K(L,M):=\text{Res}_{\lambda} (A_K(\lambda,M), \lambda^2-L).$$
It is known that the $A$-polynomial satisfies the condition $A_K(L,-M)=A_K(L,M)$, see \cite{CCGLS}. Combining \cite[Lemma 4.1]{Tr-twist} and Lemma \ref{even_M}, we have the following.

\begin{proposition}
\label{A-irreducible}
Suppose that $A_K(L,M) \in \BC[L,M]$ is irreducible and $A_K(-L,M) \not= A_K(L,M)$.   Then $R_K[L,M^2]$ is irreducible in $\BC[L,M]$ and $\deg_L(R_K)=\deg_L(A_K)$.
\end{proposition}

\subsection{The $C$-polynomial} In this subsection we define and study a relative of the $A$-polynomial, namely the $C$-polynomial. The $C$-polynomial $C_K(\ell, M)$ can be considered as a classical analog of the $\gamma$-polynomial $\gamma_K(\ell, M)$ defined in Subsection \ref{gamma}.

For simplicity, we consider the knots $K$ whose universal character ring $\fs(X_K)$ is reduced. This means that $\fs(X_K) = \BC[\chi(X_K)]$. Recall that $\ell = L^2$. Let $\ft_2 := \BC[\ell^{\pm1}, M^{\pm 1}] \subset \ft$ and $\theta_2$ the restriction of $\theta: \ft^{\sigma} \to \fs(X_K)$ on $\ft_2^{\sigma} \subset \ft^{\sigma}$. That is
$$ \theta_2 = \big( \theta \mid_{\ft_2^{\sigma}} : \ft_2^{\sigma} \to \fs(X_K) \big).$$

Recall that $\fs(X_K) = \BC[\chi(X_K)]$ has a $\ft^{\sigma}$-module structure and hence a $\BC[M^{\pm 1}]^{\sigma}$-module structure, since $\BC[M^{\pm 1}]^{\sigma}$ is a subring of $\ft^{\sigma}$. Extending the map $\theta_2: \ft_2^{\sigma} \to \fs(X_K)$ from the ground ring $\BC[M^{\pm 1}]^\sigma$ to $\BC(M)$ we get
$$
\left( \bar{\ft}_2 \overset{\bar{\theta}_2}{\longrightarrow} \overline{\fs(X_K)} \right) :=\left( \ft_2^\sigma \overset{\theta_2}{\longrightarrow} \fs(X_K) \right) \otimes _{\BC[M^{\pm 1}]^\sigma} \BC(M).
$$
The ring $\bar \ft_2=\BC(M)[\ell^{\pm1}]$ is a PID. The ideal $\bar{\fp}_2:= \ker \bar{\theta}_2 \subset \bar{\ft}_2$ can thus be generated by a  polynomial $C_K(\ell,M) \in \BC[\ell, M]$. We call $C_K(\ell,M)$ the $C$-polynomial of $K$.

We now relate the $C$-polynomial to the $A$-polynomial. 

\begin{lemma}
\label{CR}
We have $$C_{K}(\ell,M) = (\ell-1) R_K(\ell,M)$$ up to a factor depending on $M$ only.
\end{lemma}

\begin{proof}
From definition, we see that $C_{K}(\ell,M)$ is the polynomial in $\bbt_2$ of minimal $\ell$-degree such that $\bbtheta_2(C_K(\ell,M)) = 0$. Recall from Subsection \ref{char} that $B_K(L,M)$ is the polynomial in $\bbt$ of minimal $L$-degree such that $\bbtheta(B_K(L,M)) = 0$. 

Since $\ell = L^2$,  we have
\begin{eqnarray*}
C_{K}(\ell,M) &=& \text{Res}_L (B_{K}(L,M), L^2-\ell) \\
              &=&  (\ell-1) \text{Res}_L (A_{K}(L, M), L^2-\ell) \\
              &=& (\ell-1) R_K(\ell,M)
\end{eqnarray*}
up to a factor depending on $M$ only. 
\end{proof}

\begin{proposition}
\label{surj-2}
Suppose $K \subset S^3$ is a hyperbolic knot such that the universal character ring $\fs(X_K)$ is reduced, the non-abelian character variety of $\pi_1(X_K)$ is irreducible, and $A_K(L,M) \not= A_K(-L,M)$. Then the map $\bbtheta_2: \bar{\ft}_2 \to \overline{\fs(X_K)}$ is surjective.
\end{proposition}

\begin{proof}
By assumption $K$ is hyperbolic, the universal character ring $\fs(X_K)$ is reduced and the nonabelian character variety of $\pi_1(X_K)$ is irreducible. The proof of \cite[Theorem 1]{LTaj} then implies that the map $\bbtheta: \bbt \to \overline{\fs(X_K)}$ in Subsection \ref{char} is surjective. Moreover, the $B$-polynomial $B_K(L,M)$ has $L$-degree equal to the dimension of the $\BC(M)$-vector space $\overline{\fs(X_K)}$. That is $\deg_L(B_K(L,M)) = \dim_{\BC(M)}\overline{\fs(X_K)}$.

Since the non-abelian character variety of $\pi_1(X_K)$ is irreducible, the $A$-polynomial $A_K(L,M) \in \BC[L,M]$ is irreducible. Since $A_K(L,M) \not= A_K(-L,M)$, \cite[Lemma 4.1]{Tr-twist} implies that 
$
\deg_{\ell}(R_K(\ell,M)) = \deg_L(A_K(L,M)). 
$
Hence, by Lemma \ref{CR} we have 
\begin{eqnarray*}
\deg_{\ell}(C_K(\ell,M)) &=& \deg_{\ell}(R_K(\ell,M)) +1 \\
&=& \deg_L(B_K(L,M)). 
\end{eqnarray*}
This implies that $\deg_{\ell}(C_K(\ell,M)) = \dim_{\BC(M)}\overline{\fs(X_K)}$ and thus $\bbtheta_2: \bar{\ft}_2 \to \overline{\fs(X_K)}$ is surjective.
\end{proof}

\section{Localized skein modules and the AJ conjecture} \label{action}

In this section, we first recall the definition of the localized skein module $\overline{\CS(X_K)}$. We then show that under an assumption on $\overline{\CS(X_K)}$ and an extra condition, the $\gamma$-polynomial and $C$-polynomial of $K$ are related by $\ve(\gamma_K(\ell, M)) \mid C_K(\ell, M)$. See Proposition \ref{surj}.

\subsection{Localization} 

Let $D:=\CR [M^{\pm 1}]$ and $\ov{D}$ be its localization at $(1+t)$:
$$\ov{D}:=\left\{\frac{f}{g} \mid f,g \in D, \, g \not \in (1+t)   D     %g|_{t=-1} \not =0
\right\}
,$$
which is a discrete valuation ring and  is flat over $D$. The ring $D=\CR [M^{\pm 1}]$ is flat over $D^\sigma=\CR [M^{\pm 1}]^\sigma$. Hence $\bar{D}$ is flat over $D^\sigma$.

The ring $\CT_2^{\sigma} \subset \CT^{\sigma}$ is a $D^\sigma$-algebra. Let
$ \ov{\CT}_2 :=\CT_2^{\sigma} \otimes_{D^\sigma} \ov{D}$. We have
$$\ov{\CT}_2 =\left\{\sum_{k \in \BZ} a_k(M) \ell^k \mid a_k(M) \in \ov{D},~a_k=0\text{~almost always}\right\}$$
with commutation rule $a(M) \ell^{i} \cdot b(M) \ell^{k}=a(M)b(t^{4i}M) \ell^{i+k}$.

Let $\bT_{2,+}$ be the subring of $\bT_2$ consisting of all polynomials in $\ell$, i.e. polynomials like the above  with $a_k(M)=0$ if $k <0$. Although the ring $\bT_2$ is not a left PID, we have the following description of its left ideals. 

\begin{proposition} 

\label{lem.ideal}

 Suppose $I\subset \bT_2$ is a non-zero left ideal.
There are $h_0,\dots, h_{m-1} \in \bT_{2,+}$ with leading coefficients 1 and $\gamma \in \bT_{2,+}$ such that
$I$ is generated by $\{ h_0 \gamma, (1+t) h_{1} \, \gamma, \dots, (1+t)^{m-1} h_{m-1}  \gamma, (1+t)^m \gamma\}$.
Besides, $1\le \deg_{\ell}(h_{k+1}) \le  \deg_{\ell}(h_k)$ for $k= 0,\dots, m-2$; and $\gamma$ is the generator of the principal left ideal $\tilde I=I \cdot \tilde \CT_2 \subset \tilde \CT_2$.

\end{proposition}

\begin{proof}
The proof is completely similar to that of \cite[Proposition 3.3]{LTaj} and hence is omitted.
\end{proof}

\def\IM{\mathrm{Im}}

\subsection{Localized skein modules}

The skein module $\CS(X_K)$ has a $\CT^{\sigma}$-module structure, hence a $D^\sigma$-module structure since $D^\sigma$ is a subring of $\CT^{\sigma}$. Extending the map $\Theta_2: \CT_2^{\sigma} \to \CS(X_K)$ in the Subsection \ref{gamma} from the ground ring $D^\sigma$ to $\bar{D}$, we get
$$
\left( \bar{\CT}_2 \overset{\bar{\Theta}_2}{\longrightarrow} \overline{\CS(X_K)} \right) :=\left( \CT_2^\sigma \overset{\Theta_2}{\longrightarrow} \CS(X_K) \right) \otimes _{D^\sigma} \bar{D}.
$$
In \cite{LTaj}, $\overline{\CS(X_K)}$ is called the localized skein module of $X_K$. It is a module over $\bD$.

\begin{lemma}\label{lem.surj}
Suppose $\bbtheta_2$ is surjective and $\overline{\CS(X_K)}$ is finitely generated over $\bD$. Then $\bTheta_2$ is surjective.
\end{lemma}

\begin{proof}
The proof is similar to that of \cite[Lemma 3.4]{LTaj}. Consider the following commutative diagram
\be 
\label{dia.1a}
\begin{CD}  \bT_2   @>\bTheta_2  >> \overline{\CS(X_K)} \\
  @V \varepsilon VV  @V\ve VV \\
  \bbt_2     @>\bbtheta_2 >> \overline{\fs(X_K)}
\end{CD}
\ee

Suppose $\{x_1,\dots, x_d \}$ is a basis of the $\BC(M)$-vector space $\overline{\fs(X_K)}$. Let $\ov{x}_k \in \overline{\CS(X_K)}$ be a lift of $x_k$. By Nakayama's Lemma, $\{\ov x_1,\dots, \ov x_d \}$ spans $\overline{\CS(X_K)}$ over $\bD$.
Since $\bbtheta_2$ and $\ve$ in diagram \eqref{dia.1a} are surjective, each $\ov x_k$ is in the image of $\bTheta_2$. This proves  that $\bTheta_2$ is surjective.
\end{proof}

\begin{proposition}
\label{surj}
Suppose $\bbtheta_2$ is surjective and $\overline{\CS(X_K)}$ is finitely generated over $\bar{D}$. Then $\ve(\gamma_K(\ell, M)) \mid C_K(\ell, M)$ in $\BC(M)[\ell]$.
\end{proposition}

\begin{proof}
Again, the proof is similar to that of \cite[Proposition 3.5]{LTaj}. By Lemma \ref{lem.surj}, $\bTheta_2$ is surjective.
Diagram~\eqref{dia.1a} can be extended to   the following commutative diagram
with exact rows:
$$ \begin{CD}  0 @ >>> \bP_2 @ >\iota>> \bT_2   @>\bTheta_2 >> \overline{\CS(X_K)}  @>>>  0\\
@.  @ V h VV @V \varepsilon VV  @V \varepsilon VV \\
0 @ >>>     \bbp_2  @ >>>  \bbt_2     @>\bbtheta_2 >> \overline{\fs(X_K)} @>>> 0
\end{CD}
$$
where $\bar{\CP}_2 := \ker \bar{\Theta}_2 \subset \bar{\CT}_2$.

Taking the tensor product of the first row, which is an exact sequence of $\bD$-modules, with the $\bD$-algebra $\BC(M)$, we get the exact sequence
$$ \bP_2  \otimes_{\bD} \BC(M) \overset{\ve(\iota)} \longrightarrow  \bbt_2 \overset{\bbtheta_2} \longrightarrow \overline{\fs(X_K)} \to 0.$$
This implies that  $\bbp_2 = \ker(\bbtheta_2) = \IM(\ve(\iota))= \IM(\ve \circ \iota)=   h(\bP_2)$.

Suppose $\{g_k:=(1+t)^k h_k \gamma, k=0,1,\dots,m\}$ (with $h_m=1$) be a set of generators of $I=\bP_2$ as described in  Proposition \ref{lem.ideal}.
Then $h(g_k)=\ve(g_k)=0$ except possibly for $k=0$. It follows that $\bbp_2 = h(\bP_2)$ is the principal ideal generated by $\ve(g_0)$. Hence $\ve(g_0)=C_K(\ell, M) \neq 0$.
On the other hand, it is clear that $\gamma_K(\ell, M) = \gamma$. Therefore  $\ve(\gamma_K(\ell, M)) \mid \ve(g_0)= C_K(\ell, M)$.
\end{proof}

\section{Proof of main results} \label{twist}

In this last section we prove the main results, namely Theorems 1--3 and Corollary 1.

\subsection{Proof of Theorem \ref{main}} Since $K \subset S^3$ satisfies conditions $(i)$--$(v)$ of Theorem \ref{main}, Propositions \ref{surj-2} and \ref{surj} imply that
$
\ve(\gamma_K(\ell, M)) \mid C_K(\ell, M) = (\ell -1) R_K(\ell, M).$ Since $\gamma_K(L^2, M) J_K(n)=0$, by Lemma \ref{L2} we have $\gamma_K(L, t^2M^2)\BJ_K(n)=0$. This means that $\gamma_K(L, t^2M^2)$ is divisible by $\alpha_{\BJ_K}$ in  $\tilde{\CT}$. Hence 
\begin{equation}
\label{e1}
\ve(\alpha_{\BJ_{K}}) \text{ divides } (L -1) R_K(L, M^2) \text{ in } \BC(M)[L].
\end{equation} 

By \cite[Proposition 2.6]{Traj8}, $\ve(\alpha_{\BJ_{K}})$ is divisible by $L-1$.  Since the breadth of the degree of the colored Jones function $J_K(n)$ is not a linear function in $n$, the proof of \cite[Proposition 2.5]{Traj8} implies that $\alpha_{\BJ_K}$ has $L$-degree $>1$. Hence 
\begin{equation}
\label{e2}
\frac{\ve(\alpha_{\BJ_{K}})}{L-1} \text{ is a polynomial in } \BC(M)[L] \text{ of $L$-degree } \ge 1.
\end{equation} 

Since $A_K(L,M) \in \BC[L,M]$ is irreducible and $A_K(L,M) \not= A_K(-L,M)$, by Proposition \ref{A-irreducible} we have $R_{K}(L,M^2)$ is irreducible in $\BC[L,M]$. This, together with \eqref{e1} and \eqref{e2}, implies that for all knots $K \subset S^3$ satisfying conditions $(i)$--$(vi)$ of Theorem \ref{main} we have
\begin{equation}
\label{c3}
\frac{\ve(\alpha_{\BJ_{K}})}{L-1} = R_{K}(L,M^2)
\end{equation} up to a factor depending on $M$ only. 

Therefore, if the conditions that $K \in \CK$ and $r$ is an odd integer with $(r - r_{K}^+)(r - r_{K}^-)>0$ are also satisfied, then by \eqref{c3} and Proposition \ref{div} we obtain
\begin{eqnarray*}
(L-1)A_{K^{(r,2)}}(L,M) &=& (L-1)R_{K}(L,M^2)(L+M^{-2r})\\
                &= &\ve(\alpha_{\BJ_{K}})(L+M^{-2r})\\
                &=&\ve(\alpha_{K^{(r,2)}})
\end{eqnarray*}
up to a factor depending on $M$ only. This completes the proof of Theorem \ref{main}.

\subsection{Two-bridge knots} For a non-trivial adequate knot $K$, the breadth of the degree of the colored Jones polynomial $J_K(n)$ is a quadratic polynomial in $n$, by \cite[Lemma 5.4]{Li} and \cite[Proposition 2.1]{Le06}. Moreover, by \cite[Lemma 3.2]{Tr-twist} we have $K \in \CK$ and $$r_{K}^+ = 4c_K^+ \quad \text{and} \quad r_{K}^- = 4c_K^-$$
where $c^{\pm}_K$ is the number of positive/negative crossings of $K$.

For a two-bridge $K$, its skein module $\CS(X_K)$ is a free $\CR[M^{\pm 1}]^{\sigma}$-module of finite rank and its universal character ring $\fs(X_K)$ is reduced, by \cite{Le06}. Moreover, by \cite[Section 3]{Na} we have $A_K(L,\sqrt{-1}) = (L-1)^k$ for some $k \ge 1$. This implies that $A_K(L,M)$ contains at least one term $L^i M^k$, where $i$ is odd. In particular, we have $A_K(L,M) \not= A_K(-L,M)$. 

Theorem \ref{main:two-bridge} now follows from Theorem \ref{main}. 

Double twist knots $J(k,l)$, with $k \not= l$, have irreducible non-abelian character varieties by \cite{MPL}. Hence Corollary 1 follows from Theorem \ref{main:two-bridge}.

\subsection{Pretzel knots} Let $K(m)$ denote the $(-2,3,2m+3)$-pretzel knot. Then $K(m)$ belongs to the set $\CK$ by \cite{KT2}. It is known that $K(m)$ is a torus knot if $|m| \le 1$, and is a hyperbolic knot if $|m| \ge 2$. Since $K(-2)$ is the twist knot $5_2$, Theorem \ref{main:pretzel} for $K(-2)$ follows from Theorem \ref{main:two-bridge}. Hence we consider the two cases $m \ge 2$ and $m \le -3$ only.

For $m \ge 2$ we have
\begin{eqnarray*}
d_-[J_{K(m)}(n)] &=& (m + 5/2) (n-1),\\
d_+[J_{K(m)}(n)] &=& \left( \frac{5}{2} + m + \frac{1}{4m} \right) n^2 + \left( \frac{1}{2m} - \frac{1}{2} \right) n - \left( 3 + \frac{3m}{4} - \frac{1}{4m} \right) - m r^2_{n-1},
\end{eqnarray*}
where $r_n$ is a periodic sequence with $|r_n| \le 1/2$. Hence 
$$r^{-}_{K(m)} = 0 \quad \text{and} \quad r^{+}_{K(m)} \le \frac{33m}{4} + \frac{3}{m} + 19.$$

For $m \le -3$ we have
\begin{eqnarray*}
d_+[J_{K(m)}(n)] &=& 5n^2/2 + (1 + m) n - (7/2 + m),\\
d_-[J_{K(m)}(n)] &=& \frac {2(m+2)^2} {2m+3} \, n^2 + b(n) n - (6m+17)/8 - (m + 3/2) r^2_{n-1},
\end{eqnarray*}
where $r_n$ is a periodic sequence with $|r_n| \le 1/2$, and $b(n) = \begin{cases} 1/2   &\mbox{if } (2m+3) \nmid n, \smallskip \\ 
\frac {2m+1} {2( 2m+3)} & \mbox{if } (2m+3) \mid n. \end{cases}$
Hence
$$r^{-}_{K(m)} \le \frac{33m}{4} + \frac{6}{2m+3} + \frac{171}{8} \quad \text{and} \quad r^{+}_{K(m)} = 20.$$

For the $(-2,3,2m+3)$-pretzel knot $K(m)$, its localized skein module $\overline{\CS(X_{K(m)})}$ is finitely generated over $\bD$ and its universal character ring $\fs(X_{K(m)})$ is reduced, by \cite{LTaj}. The breadth of the degree of the colored Jones polynomial $J_{K(m)}(n)$ is not a linear function in $n$, by \cite{Ga-slope}. Moreover, if $\gcd(3,m)=1$ then the non-abelian character variety of $\pi_1(X_{K(m)})$ is irreducible, by \cite{Ma} (see also \cite{Tr-agt}). 

If $P(L,M)=\sum_{i,k} a_{ik} L^i M^k \in \BC[L,M]$ then its Newton polygon is the smallest convex set in the plane containing all integral lattice points $(i,k)$ for which $a_{ik} \not=0$. 

By \cite{Ma} (see also \cite{GM}), the Newton polygon of the $A$-polynomial of the $(-2,3,2m+3)$-pretzel knot $K(m)$ is a hexagon with the following vertex set
\begin{equation*}
\begin{split}
\{ & (0, 0), (1, 16), (m - 1, 4 (m^2 + 2 m - 3)), (2 m + 1, 
  2 (4 m^2 + 10 m + 9)), \\ & (3 m - 1, 2 (6 m^2 + 14 m - 5)), (3 m, 
  2 (6 m^2 + 14 m + 3))\}
\end{split}
\end{equation*}
if $m>1$,
\begin{equation*}
\begin{split}
\{ & (-3 m - 4, 0), (-3 (1 + m), 10), (-3 - 2 m, 4 (3 + 4 m + m^2)), \\
& (-m,  2 (4 m^2 + 16 m + 21)), (0, 4 (3 m^2 + 12 m  + 11)), (1, 
  6 (2 m^2 + 8 m + 9))\}
\end{split}
\end{equation*}
if $m<-2$
and
$$
\{(0, 0), (1, 0), (2, 4), (1, 10), (2, 14), (3, 14)\}
$$
if $m=-2$. This implies that $A_{K(m)}(L,M)$ contains at least one term $L^i M^k$, where $i$ is odd. In particular, we have $A_{K(m)}(L,M) \not= A_{K(m)}(-L,M)$. 

Theorem \ref{main:pretzel} now follows from Theorem \ref{main}.

\end{document}